\documentclass{article}
\usepackage[utf8]{inputenc}

\usepackage{amssymb,amsmath,amsthm,mathabx}
\usepackage{wasysym}
\usepackage{graphicx}
\usepackage{proof}
\usepackage{cite}
\usepackage[UKenglish]{babel}
\usepackage{color}
\usepackage[all]{xy} %xymatrix
\usepackage{hyperref}
\usepackage{times}

 \newcommand{\bis}{\mathrel{\mathchoice%
{\raisebox{.3ex}{$\,
  \underline{\makebox[.7em]{$\leftrightarrow$}}\,$}}%
{\raisebox{.3ex}{$\,
  \underline{\makebox[.7em]{$\leftrightarrow$}}\,$}}%
{\raisebox{.2ex}{$\,
  \underline{\makebox[.5em]{\scriptsize$\leftrightarrow$}}\,$}}%
{\raisebox{.2ex}{$\,
  \underline{\makebox[.5em]{\scriptsize$\leftrightarrow$}}\,$}}}}
  
\hypersetup{
   colorlinks,%
   citecolor=blue,%
   filecolor=black,%
   linkcolor=red,%
   urlcolor=black
 }%超链接的颜色
 
 \newcommand{\lr}[1]{\langle #1 \rangle}
\newcommand{\lra}{\leftrightarrow}

\newcommand{\I}{\ensuremath{\textit{I}}}

\newcommand{\M}{\ensuremath{\mathcal{M}}}

\newcommand{\BP}{\textbf{P}}

\renewcommand{\phi}{\varphi}

\newcommand{\weg}[1]{}

\theoremstyle{definition}

\newtheorem{theorem}{Theorem}
\newtheorem{lemma}[theorem]{Lemma}
\newtheorem{definition}[theorem]{Definition}

\newtheorem{remark}[theorem]{Remark}
\newtheorem{proposition}[theorem]{Proposition}

\newtheorem{fact}[theorem]{Fact}

\title{Logics of False Belief and Radical Ignorance}

\author{Jie Fan\\
\small Institute of Philosophy, Chinese Academy of Sciences;\\
\small School of Humanities, University of Chinese Academy of Sciences  \\
\small \texttt{jiefan@ucas.ac.cn}}
\date{}

\begin{document}

\maketitle

\begin{abstract}

In the literature, the question about how to axiomatize the transitive logic of false belief is thought of as hard and left as an open problem.
In this paper, among other contributions, we deal with this problem. In more details, although the standard doxastic operator is undefinable with the operator of false belief, the former is {\em almost definable} with the latter. On one hand, the involved almost definability schema guides us to find the desired core axioms for the transitive logic and the Euclidean logic of false belief. On the other hand, inspired by the schema and other considerations, we propose a suitable canonical relation, which can uniformly handle the completeness proof of various logics of false belief, including the transitive logic. We also extend the results to the logic of radical ignorance, due to the interdefinability of the operators of false belief and radical ignorance.
    %In this paper, we propose a logic of (first-order) ignorance whether and Rumsfeld ignorance. We axiomatize this logic over various frame classes, in which the completeness of the minimal system uses a nontrivial inductive method. We then apply our work to Fine~\cite{Fine:2018}. As we will show, within our framework, we can simplify and extend some of the results in~\cite{Fine:2018}, and obtain some results that Fine did not notice. This may help us see more clearly the relationship among first-order ignorance, second-order ignorance, and Rumsfeld ignorance.
\end{abstract}


\section{Introduction}

%This article is intended to provide a comparative analysis of the formal representation of three notions, namely false belief, factive ignorance, and radical ignorance, and build a bridge among their logics.
%This article is intended to provide a comparative analysis of the formal representation of two notions, namely false belief and radical ignorance, and build a bridge among their logics.
The prime motivation of this paper is to deal with an open problem about how to axiomatize the transitive logic of false belief.

The discussion about false belief dates back to Plato~\cite[Sec.~7]{sep-plato-theaetetus}. This notion is related to a well-known distinction between knowledge and belief: knowledge must be true, but belief can be false, that is, there are false beliefs. Besides, this notion is popular in the field of cognitive science, see e.g.~\cite{Onishietal:2005}.

The first logical study on false belief is done by Steinsvold in~\cite{Steinsvold:falsebelief}, where a logic of false belief is proposed that has the operator $W$ as a sole primitive modality. There, $W\phi$ is read ``one is wrong about $\phi$'', or ``$\phi$ is a false belief'', meaning that $\phi$ is false though believed. This logic is axiomatized over the class of all frames and over various frame classes, and some results of frame definability are presented. The logic is then interpreted over neighborhood semantics in~\cite{GilbertVenturi:2017} and~\cite{Fan:2022neighborhood}, and analyzed in the intuitionistic setting in~\cite{Witczak:2022}.

Although both~\cite[Sec.~5]{Steinsvold:falsebelief} and~\cite[Sec.~2.4]{GilbertVenturi:2017} spend a whole section on discussing  the problem about how to axiomatize the transitive logic of false belief, they think it is difficult. To our best knowledge, this open  problem has not been solved until now. Two difficulties are involved. One is how to find the desired core axiom, and the other is how to define a suitable canonical relation.%Moreover, even though there are many axiomatization results in~\cite{Steinsvold:falsebelief}, some of them are problematic, since the soundness results are {\em not} correct.

In this paper, we will show that although the standard doxastic operator is undefinable with the operator of false belief, the former is almost definable with the latter. On one hand, the involved almost definability schema guides us to find the desired core axioms for the transitive logic and the Euclidean logic of false belief. On the other hand, inspired by the schema and other considerations, we propose a suitable canonical relation, which can uniformly handle the completeness proof of various logics of false belief, including the transitive logic, thereby solving the open problem in~\cite{Steinsvold:falsebelief}.

%Inspired by the involved almost definability schema, we obtain the desired core axiom and also a suitable canonical relation (thus a suitable canonical model), then we axiomatize the logic of false belief over transitive frames, thereby solving the open question in~\cite{Steinsvold:falsebelief}. %We also correct the incorrect soundness results in the cited paper and give the right axiomatizations.

%The logic of false belief is introduced by Steinsvold in~\cite{Steinsvold:falsebelief}. Intuitively, one has a {\em false belief} for a proposition $\phi$, or one {\em is wrong about} $\phi$, if $\phi$ is false though one believes that $\phi$. Formally, $$\M,s\vDash W\phi\text{~iff~}\M,s\nvDash\phi\text{~and for all~}t,\text{~if~}sRt\text{~then~}\M,t\vDash\phi.$$

Moreover, we extend the results to the logic of radical ignorance. The notion of radical ignorance is proposed in~\cite{Fano:2020working}, to (hopefully) adequately express the important properties in the phenomenon of the Dunning–Kruger effect. This notion is formalized by using a {\bf KT4-B4} framework, in which the epistemic accessibility relation $R_K$ is reflexive and transitive, the doxastic accessibility relation $R_B$ is transitive, and $R_B$ is included in $R_K$. An agent is radically ignorant about $\phi$ iff the agent does not know $\phi$ and also does not know $\neg\phi$,\footnote{The original wording on page 611 of~\cite{Fano:2020working} is ``the agent does not know both $\phi$ and $\neg\phi$'', where there is a danger of misunderstanding because of a scope ambiguity.} and either the agent believes $\phi$ but it is the case that $\neg\phi$, or it is the case that $\phi$ but the agent believes $\neg\phi$; in symbol, $I_R\phi=_{df}((\neg K\phi\land\neg K\neg \phi)\land((B\phi\land\neg \phi)\vee (B\neg \phi\land\phi)))$, where $K$ and $B$ are standard opeartors of knowledge and belief, resepctively. Under very natural assumptions, namely $K\phi\to \phi$, $K\phi\to B\phi$ and $\neg(B\phi\land B\neg\phi)$, the definition can be simplified to the following: $I_R\phi=_{df}((B\phi\land\neg\phi)\vee (B\neg\phi\land\phi))$.\footnote{There is an error on page 611 of~\cite{Fano:2020working}: it says this simplification can be done in the framework {\bf KT4-B4}, but the serial axiom for the belief operator, viz. $\neg(B\phi\land B\neg\phi)$, is lacking. So it is {\bf KT4-BD4} (instead of {\bf KT4-B4}) frames, where the doxastic accessibility relation is serial as well, that are the frames which the framework of the cited paper is actually based on.\label{fn.remark}} However, there have been no formal systems characterizing the notion of radical ignorance.

A related work in the literature is the investigation on the notion of reliable belief in~\cite{Fan:2022sane}, with different motivations though. Note that the operator $\mathcal{R}$ of reliable belief is the negation of $I_R$, since $\mathcal{R}\phi$ is equivalent to $(B\phi\to\phi)\land (B\neg\phi\to\neg\phi)$. The minimal logic and the serial logic of reliable belief are axiomatized there. However, the canonical model there does not apply to the transitive logic of reliable belief, thus not to the transitive logic of radical ignorance either.

As we will see, the operators of radical ignorance and false beliefs are interdefinable with each other. This may indicate that one can translate the results about false belief into those about radical ignorance via the translation induced by the definability of the operator $W$ in terms of the operator of radical ignorance. Unfortunately, this holds for all but the minimal proof system.\footnote{Note that this is not new. Even if two operators are interdefinable with each other, it is {\em not} necessary that the axiomatizations of the logic with one operator as a sole modality can be obtained from those of the logic with the other operator as a sole modality via the translation induced by the interdefinability of the operators. For instance, although the necessity operator $\Box$ and the dyadic contingency operator are interdefinable with each other, the serial logic of dyadic contingency cannot be obtained from the serial system {\bf KD} via the translation induced by the interdefinability of the operators, see~\cite[p.~214]{Fan:2023}.} We will illustrate this with the axiomatizations of the logic of radical ignorance over all frames and over serial and transitive frames.

%van Woudenberg~\cite{Woudenberg:2009} argues that false belief is a sufficient condition for ignorance.

The remainder of this paper is organized as follows. After introducing the language and semantics of the logic of false belief, we propose an almost definability schema (Sec.~\ref{sec.syntaxandsemantics}). Then we compare the expressive powers of the logic of false belief and standard doxastic logic, and investigate the frame definability of the former (Sec.~\ref{sec.expressivityandframedefinability}). Sec.~\ref{sec.axiomatizations} axiomatizes the logic of false belief over various frame classes, including the transitive logic (Sec.~\ref{sec.transitivelogic}) specially, thereby solving an open problem raised in~\cite{Steinsvold:falsebelief}. The canonical relation here is inspired by the aforementioned almost definability schema and other considerations. Moreover, the desired core axioms for the transitive logic and the Euclidean logic of false belief are obtained from the familiar axioms via a translation induced by the almost definability schema. Last but not least, we axiomatize the logic of radical ignorance (Sec.~\ref{sec.radicalignorance}). After briefly reviewing the language and semantics of the logic of radical ignorance, we note that the operators of radical ignorance and of false beliefs are interdefinable with each other. Then we axiomatize the minimal logic (Sec.~\ref{sec.minimallogic}) and the serial and transitive logic (Sec.~\ref{sec.d4-radicalignornce}) of radical ignorance.
 
%one is {\em factively ignorant} about $\phi$, if $\phi$ is true at the current state but false at all other accessible states; one is {\em radically ignorant about} $\phi$ iff the agent does not know $\phi$ and also does not know $\neg\phi$, and either the agent believes $\phi$ but it is the case that $\neg\phi$, or it is the case that $\phi$ but the agent believes $\neg\phi$.

\medskip

\section{False belief}
%\section{Logics of false beliefs}

\subsection{Syntax and Semantics}\label{sec.syntaxandsemantics}

Throughout the paper, we assume $\BP$ to be a nonempty set of propositional variables. We first define a logical language including both false belief and belief operators. The language of the standard doxastic logic and the language of the logic of false belief can be viewed as two fragment of this language. It is the latter that is our main focus in the rest of the paper.  
\begin{definition}
The language $\mathcal{L}(W,\Box)$ is defined recursively as follows.
\[
\begin{array}{lll}
   \phi &::= & p\mid \neg\phi\mid (\phi\land\phi)\mid W\phi\mid \Box\phi, \\
\end{array}
\]
where $p\in \BP$. The {\em language of the logic of false belief}, denoted $\mathcal{L}(W)$, is the fragment of $\mathcal{L}(W,\Box)$ without the construct $\Box\phi$. The {\em language of standard doxastic logic}, denoted $\mathcal{L}(\Box)$, is the fragment of $\mathcal{L}(W,\Box)$ without the construct $W\phi$.
\end{definition}

Other connectives are defined as usual. The formula $W\phi$ is read ``$\phi$ is a false belief of the agent'', or ``the agent is wrong about $\phi$''.\footnote{In a deontic setting, $W\neg\phi$ (that is, $\Box\neg\phi\land\phi$) is read ``$\phi$ ought not to be done but done'', which expresses some kind of vice: one did what one ought not to have done. In a metaphysical setting, $W\neg\phi$ is read ``$\phi$ is strongly accidental'', c.f.~\cite{PanYang:2017}.} The language is interpreted on models.
\begin{definition}\label{def.structure}
    A {\em model} is a triple $\M=\lr{S,R,V}$, where $S$ is a nonempty set of states, $R$ is a binary relation over $S$, called `accessibility relation', and $V$ is a valuation. A {\em frame} is a model without valuations; in this case, we also say that the model is based on the frame. We use ${\sim}sRt$ to mean that ``it does not hold that $sRt$''.
\end{definition}

\begin{definition}[Semantics]\label{def.semantics}
Given a model $\M=\lr{S,R,V}$ and a state $s\in S$, the semantics of $\mathcal{L}(W,\Box)$ is defined recursively as follows.
\[
\begin{array}{|lll|}
\hline
    \M,s\vDash p & \text{iff} & s\in V(p)\\
    \M,s\vDash\neg\phi & \text{iff} & \M,s\nvDash\phi\\
    \M,s\vDash\phi\land\psi&\text{iff} & \M,s\vDash\phi\text{ and }\M,s\vDash\psi\\
    
    \M,s\vDash W\phi & \text{iff} & \M,s\nvDash\phi\text{ and for all }t, \text{ if }sRt \text{ then }\M,t\vDash\phi\\ 
    \M,s\vDash\Box\phi & \text{iff} & \text{for all }t, \text{ if }sRt \text{ then }\M,t\vDash\phi.\\
\hline
\end{array}
\]

If $\M,s\vDash\phi$, we say that $\phi$ is {\em true} in $(\M,s)$, and sometimes write $s\vDash\phi$ if $\M$ is clear. If for all frames $\mathcal{F}$ in $F$, for all models $\M$ based on $\mathcal{F}$, for all $s$ in $\M$ we have $\M,s\vDash\phi$, then we say that $\phi$ is {\em valid on $F$} and write $F\vDash\phi$; when $F$ is the class of all frames, we say $\phi$ is {\em valid} and write $\vDash\phi$. The notions for a set of formulas are defined similarly.
\end{definition}

From the above semantics, it follows easily that $W$ is definable in terms of $\Box$, as $\vDash W\phi\lra (\Box\phi\land\neg\phi)$. In contrast, as we will show in Prop.~\ref{prop.expressivity-lw}, the converse does not hold, since $\Box$ is {\em not} definable in terms of $W$ in various classes of models.

Let $R(s)=\{t\mid sRt\}$. The semantics of $W$ can be rewritten as follows:
\[
\begin{array}{lll}
 \M,s\vDash W\phi & \text{iff} & \M,s\nvDash\phi\text{ and }R(s)\vDash\phi.\\
\end{array}
\]

Although $\Box$ is undefinable with $W$, we have the following important schema, which says that $\Box$ is {\em almost definable} in terms of $W$. We call it `Almost Definability Schema'.

\begin{proposition}\label{prop.almost-definability-schema}
$\vDash W\psi\to (\Box\phi\lra W(\phi\land\psi)).$
\end{proposition}

\begin{proof}
    Let $\M=\lr{S,R,V}$ be a model and $s\in S$. Suppose that $\M,s\vDash W\psi$, to show that $\M,s\vDash\Box\phi\lra W(\phi\land\psi)$.

    First, assume that $\M,s\vDash \Box\phi$, to prove that $\M,s\vDash W(\phi\land\psi)$. By assumption, we infer that $R(s)\vDash \phi$. By supposition, we have $\M,s\nvDash \psi$ and $R(s)\vDash\psi$, thus $\M,s\nvDash\phi\land\psi$ and $R(s)\vDash\phi\land\psi$. This implies that $\M,s\vDash W(\phi\land\psi)$.

    Conversely, assume that $\M,s\vDash W(\phi\land\psi)$, then $R(s)\vDash \phi\land\psi$. So $R(s)\vDash\phi$, and therefore $\M,s\vDash \Box\phi$.
\end{proof}

This schema is very important, since it not only guides us to find out the desired core axioms of transitive logic and Euclidean logic for $\mathcal{L}(W)$, it also motivates the canonical relation in
the construction of canonical model for $\mathcal{L}(W)$. With this relation we can show the completeness of all axiomatizations uniformly, as we will see below.

\subsection{Expressivity and Frame Definability}\label{sec.expressivityandframedefinability}

In this part, we investigate the expressive power and frame definability of $\mathcal{L}(W)$. To begin with, we have the following useful observation, which follows directly from the semantics of $W$.
\begin{fact}\label{fact.imp}
For all $\phi$, $W\phi$ is false in each reflexive state.    
\end{fact}

\begin{proposition}\label{prop.reflexive}
For any reflexive frames $\mathcal{F}$ and $\mathcal{F}'$, for any $\phi\in\mathcal{L}(W)$, $\mathcal{F}\vDash\phi$ iff $\mathcal{F}'\vDash\phi$.
\end{proposition}

\begin{proof}
Let $\mathcal{F}=\lr{S,R}$ and $\mathcal{F}'=\lr{S',R'}$ be reflexive frames, and let $\phi\in\mathcal{L}(W)$.

Suppose that $\mathcal{F}\nvDash \phi$, then there is a valuation $V$ and a state $s$ such that $\lr{F,V},s\nvDash\phi$. Since $S'\neq \emptyset$, we may assume that $s'\in S$. Define a valuation $V'$ on $\mathcal{F}'$ as follows: $s'\in V'(p)$ iff $s\in V(p)$ for all $p\in\BP$. Since $\mathcal{F}$ and $\mathcal{F}'$ are both reflexive, both $s$ and $s'$ are reflexive. By Fact~\ref{fact.imp}, this means that all $W\phi$ are false in both states. Then by induction on $\psi\in\mathcal{L}(W)$, we can show that $\lr{\mathcal{F},V},s\vDash\psi$ iff $\lr{\mathcal{F}',V'},s'\vDash\psi$. Hence $\lr{\mathcal{F}',V'},s'\nvDash\psi$, and therefore $\mathcal{F}'\nvDash\phi$. The converse is similar.
\end{proof}

It turns out that $\mathcal{L}(W)$ is less expressive than $\mathcal{L}(\Box)$ on various classes of models.
\begin{proposition}\label{prop.expressivity-lw}
$\mathcal{L}(W)$ is less expressive than $\mathcal{L}(\Box)$ on the class of $\mathcal{S}5$-models. As a consequence, $\mathcal{L}(W)$ is less expressive than $\mathcal{L}(\Box)$ on the class of $\mathcal{K}$-models, $\mathcal{D}$-models, $\mathcal{T}$-models, 4-models, $\mathcal{B}$-models, 5-models, $\mathcal{D}4$-models, $\mathcal{D}45$-models.
\end{proposition}

\begin{proof}
As mentioned above, $W$ is definable in terms of $\Box$, thus $\mathcal{L}(\Box)$ is at least as expressive as $\mathcal{L}(W)$. For the strict part, consider the following $\mathcal{S}5$-models:

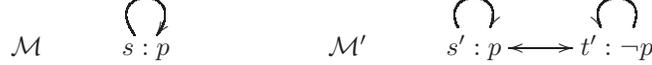
\[
\xymatrix{\M&s:p\ar@(ul,ur)&& \M'& s':p\ar@(ul,ur)\ar[r]&t':\neg p\ar[l]\ar@(ur,ul)\\}
\]

Using Fact~\ref{fact.imp}, we can show by induction that for all $\phi\in\mathcal{L}(W)$, $\M,s\vDash \phi$ iff $\M',s'\vDash \phi$. Thus $(\M,s)$ and $(\M',s')$ cannot be distinguished by $\mathcal{L}(W)$.

However, these two models can be distingished by $\mathcal{L}(\Box)$, since $\M,s\vDash\Box p$ but $\M',s'\nvDash\Box p$.
\end{proof}

It may be natural to ask if there is a class of frames where $\Box$ is definable in terms of $W$. The answer is positive. We borrow a notion of {\em narcissistic} from~\cite[Def.~2.1]{Steinsvold:falsebelief}. Call $s$ {\em narcissistic} if and only if $s$ relates to itself and only to itself. Call a frame {\em narcissistic} if all the worlds are narcissistic; that is,
$$\forall x(xRx\land \forall y(xRy\to x=y)).$$
\begin{proposition}
    On the class of narcissistic frames, $\Box$ is definable in terms of $W$. As a consequence, $\mathcal{L}(W)$ and $\mathcal{L}(\Box)$ are equally expressive on the class of narcissistic models.
\end{proposition}

\begin{proof}
    Let $F_{nar}$ be the class of narcissistic frames. It is straightforward to verify that $F_{nar}\vDash \Box\phi\lra \phi$. This means that on the frame classes in question, $\Box$ is already definable in the language of propositional logic; needless to say, $\Box$ is definable in terms of $W$.
\end{proof}

\begin{remark}
In~\cite[Sect.~1.4]{Gilbertetal:2022}, the authors compare the expressive power of $\mathcal{L}(I)$ and $\mathcal{L}(\Box)$.\footnote{In~\cite{Gilbertetal:2022}, $\mathcal{L}(I)$ is the language of the logic of factive ignorance that has the operator $I$ of factive ignorance as a sole primitive modality. Boolean formulas are interpreted as usual, and $I\phi$ is interpreted as follows: given a model $\M=\lr{S,R,V}$ and a state $s\in S$, $\M,s\vDash I\phi$ iff $\M,s\vDash\phi$ and for all $t\in S$, if $sRt$ and $s\neq t$, then $\M,t\nvDash\phi$.} It turns out that neither of $I$ and $\Box$ is, in general, definable in terms of the other. In particular, it is shown in~\cite[Coro.~1.31]{Gilbertetal:2022} that the indefinability of $\Box$ in terms of $I$ applies to a wide variety of frame classes. In the meanwhile, the authors ask whether there exist any interesting classes of frames in which $\Box$ is definable in terms of $I$ and they think the answer is negative (see~\cite[p.~878]{Gilbertetal:2022}). However, the answer is actually positive, since on the class of narcissistic frames, $\Box$ is definable in terms of $I$. The same class of frames also establishes the definability of $I$ in terms of $\Box$, since as one may show, $F_{nar}\vDash I\phi\lra \phi$, where $F_{nar}$ is the class of narcissistic frames. 
\end{remark} 

The following result is shown in~\cite[Thm.~4.8]{Steinsvold:falsebelief}, where the proof is based on a canonical model. Here we give a much simpler proof, without need of canonical models.
\begin{proposition}\label{prop.frameundefinability}
The following frame properties are all undefinable in $\mathcal{L}(W)$:\footnote{To say a frame property P is {\em definable} in a logic $\mathcal{L}$, if there is a set $\Gamma$ of $\mathcal{L}$-formulas such that for all
frames $\mathcal{F}$, $\mathcal{F}\vDash \Gamma$ iff $\mathcal{F}$ has P.}
\begin{itemize}
    \item[(1)] Transitivity,
    \item[(2)] Euclideanness,
    \item[(3)] Symmetry,
    \item[(4)] weak connectedness $\forall x\forall y\forall z((xRy\land xRz)\to (yRz\vee y=z\vee zRy))$,
    \item[(5)] weak directedness $\forall x\forall y\forall z((xRy\land xRz)\to \exists v(yRv\land zRv))$,
    \item[(6)] partial functionality $\forall x\forall y\forall z((xRy\land xRz)\to z=y)$,
    \item[(7)] narcissism $\forall x(xRx\land \forall y(xRy\to x=y))$,
    \item[(8)] partial narcissism $\forall x \forall y(xRy\to x=y)$.  
\end{itemize}
\end{proposition}

\begin{proof}
Consider the following frames:

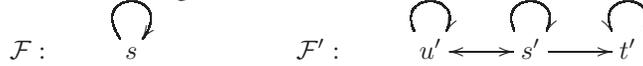
\[
\xymatrix{\mathcal{F}: & s\ar@(ul,ur) && \mathcal{F}': & u'\ar@(ul,ur)\ar[r] & s'\ar@(ul,ur)\ar[r]\ar[l] & t'\ar@(ur,ul)\\}
\]

One may check that for any frame property $P$ in (1)-(8), it is {\em not} the case that $\mathcal{F}$ has $P$ iff $\mathcal{F}'$ has $P$. However, since $\mathcal{F}$ and $\mathcal{F}'$ are both reflexive, by Prop.~\ref{prop.reflexive}, for all $\phi\in\mathcal{L}(W)$, we have that $\mathcal{F}\vDash \phi$ iff $\mathcal{F}'\vDash\phi$. This entails that the frame properties in question are all undefinable in $\mathcal{L}(W)$.
\end{proof}

\subsection{Axiomatizations}\label{sec.axiomatizations}
This section presents the axiomatizations of $\mathcal{L}(W)$ over various frame classes.

\subsubsection{Minimal Logic}\label{subsec.minimallogic}

\begin{definition}
The minimal logic of $\mathcal{L}(W)$, denoted ${\bf K^W}$, consists of the following axioms and inference rules:
\[
\begin{array}{ll}
   \text{A0}  & \text{all instances of propositional tautologies} \\
   \text{A1}  & W\phi\to\neg\phi\\
   \text{A2} & W\phi\land W\psi\to W(\phi\land\psi)\\
   \text{MP} & \dfrac{\phi~~~\phi\to\psi}{\psi}\\
   \text{R1} & \dfrac{\phi\to\psi}{W\phi\land\neg\psi\to W\psi}\\
   %\text{R1} & \dfrac{\phi\to\psi}{W(\phi\land\chi)\land\neg\psi\to W\psi}\\
\end{array}
\]
\end{definition}

The notions of theorems, provability, and derivation are defined as usual.

This system is called ${\bf S^W}$ in~\cite{Steinsvold:falsebelief}. Intuitively, axiom A1 says that false beliefs are false, A2 says that false beliefs are closed under conjunction: the conjunction of two false beliefs are still a false belief. The rule R1 stipulates the almost monotonicity of the false belief operator. It is shown in~\cite[Thm.~3.1]{Steinsvold:falsebelief} that the substituition rule of equivalents for the operator $W$, i.e. $\phi\lra\psi~/~W\phi\lra W\psi$, denoted $\text{REW}$, is admissible in ${\bf K^W}$. Moreover, we have the following.%Note that our inference rule R1 is stronger than R1 in~\cite[p.~247]{Steinsvold:falsebelief}, in that by letting $\chi$ be $\top$, we obtain the latter R1.

\begin{proposition}\label{prop.imp}
The following rule is admissible in ${\bf K^W}$:
$$\dfrac{\phi\to\psi}{W(\phi\land\chi)\land\neg\psi\to W\psi}.$$
\end{proposition}

\begin{proof}
We have the following proof sequence in ${\bf K^W}$.
\[
\begin{array}{lll}
   (1) & \phi\to\psi  & \text{hypothesis} \\
   (2) & W\phi\land\neg\psi\to W\psi  & (1),\text{R1}\\
   (3) & \phi\land\chi\to\phi & \text{A0}\\
   (4) & W(\phi\land\chi)\land\neg \phi\to W\phi & (3),\text{R1}\\
   (5) & W(\phi\land\chi)\land\neg\phi\land\neg\psi\to W\psi & (2),(4)\\
   (6) & W(\phi\land\chi)\land \neg\psi\to W\psi & (1),(5)\\
\end{array}
\]
\end{proof}

We can generalize the above proposition to the following.
\begin{proposition}\label{prop.admissiblerule} Let $n\in\mathbb{N}$.
If $\vdash \chi_1\land\cdots\land\chi_n\to\phi$, then $\vdash W(\chi_1\land\psi)\land\cdots\land W(\chi_n\land\psi)\land\neg\phi\to W\phi$.
\end{proposition}

\begin{proof}
    By induction on $n\in\mathbb{N}$. The case $n=0$ is obvious. The case $n=1$ is shown as in Prop.~\ref{prop.imp}.

    Now suppose that the statement holds for $m$ (IH), to show that it also holds for $m+1$. For this, we have the following proof sequence:
    \[
    \begin{array}{lll}
       (1) & \chi_1\land\cdots\land \chi_{m+1}\to \phi & \text{premise} \\
       (2) & (\chi_1\land\chi_2)\land\cdots\land \chi_{m+1}\to \phi & (1)\\
       (3) & W((\chi_1\land\chi_2)\land\psi)\land\cdots\land W(\chi_{m+1}\land\psi)\land \neg\phi\to W\phi & (2),\text{IH}\\
       (4) & W(\chi_1\land\psi)\land W(\chi_2\land\psi)\to W((\chi_1\land\chi_2)\land\psi) & \text{A2}\\
       (5) & W(\chi_1\land\psi)\land\cdots\land W(\chi_{m+1}\land\psi)\land\neg\phi\to W\phi & (3),(4)\\
    \end{array}
    \]
\end{proof}

The completeness of ${\bf K^W}$ is shown via the construction of a canonical model.
\begin{definition}\label{def.cm}
The {\em canonical model for ${\bf K^W}$} is $\M^c=\lr{S^c,R^c, V^c}$, where
\begin{itemize}
    \item $S^c=\{s\mid s\text{ is a maximal consistent set for }{\bf K^W}\}$,
    \item for all $s,t\in S^c$, $R^c$ is defined as follows: 
    \begin{itemize}
        \item if $W\psi\in s$ for {\em no} $\psi$, then $sR^ct$ iff $s=t$, and
        \item if $W\psi\in s$ for {\em some} $\psi$, then $sR^ct$ iff for all $\phi$, if $W(\phi\land\psi)\in s$, then $\phi\in t$.
    \end{itemize}
\item $V^c(p)=\{s\in S^c\mid p\in s\}$.
\end{itemize}
\end{definition}

It is worth noting that the above definition of $T^c$ is inspired by Almost Definability Schema (Prop.~\ref{prop.almost-definability-schema}). Recall that in the construction of the canonical model of standard doxastic logic, the canonical relation $R^c$ is usually defined as follows: $sR^ct$ iff for all $\phi$, if $\Box\phi\in s$, then $\phi\in t$. According to Almost Definability Schema, $\Box\phi\in s$ can be replaced by $W(\phi\land\psi)\in s$ provided that $W\psi\in s$ for some $\psi$. This is similar to the case for minimal contingency logic~\cite{Fanetal:2015}. 

However, unlike the case for minimal contingency logic~\cite{Fanetal:2015}, here ``$W\psi\in s$ for some $\psi$'' should be a precondition, instead of a conjunction, of the aforementioned replacement.\footnote{In other words, in the case of $W\psi\in s$ for some $\psi$, we replace $\Box\phi\in s$ with $W(\phi\land\psi)\in s$ in the definition of the canonical relation of the canonical model for standard doxastic logic.} Moreover, if this precondition is not satisfied, then $s$ can and only can access itself. As we will see, the case-by-case definition enables us to prove the completeness of the minimal system and its extensions, which however cannot be done if we use ``$W\psi\in s$ for some $\psi$'' as a conjunction (as the reader may verify).

Also notice that our definition differs from the canonical relation in~\cite[Def.~4.2]{Steinsvold:falsebelief} in that we have $W(\phi\land\psi)\in s$ instead of $W\phi\in s$. Besides, as already mentioned above, our definition is motivated by Almost Definability Schema. As we will see, the slight distinction enables us to show the completeness of the transitive system of $\mathcal{L}(W)$ (Sec.~\ref{sec.transitivelogic}), which cannot be done with $W\phi\in s$ instead.

The following result states that the truth lemma holds for ${\bf K^W}$.
\begin{lemma}
For all $\phi\in\mathcal{L}(W)$ and for all $s\in S^c$, we have
$$ \M^c,s\vDash\phi \text{~iff~}\phi\in s.$$
\end{lemma}

\begin{proof}
By induction on $\phi$. We only consider the case $W\phi$.

`If': suppose that $W\phi\in s$, to show that $\M^c,s\vDash W\phi$. By supposition and axiom A1, $\neg\phi\in s$, and thus $\phi\notin s$. By IH, we have $\M^c,s\nvDash\phi$. Now let $t\in S^c$ such that $sR^ct$, by IH, it suffices to show that $\phi\in t$. By definition of $T^c$ and supposition, we infer that for all $\chi$, if $W(\chi\land\phi)\in s$, then $\chi\in t$. By letting $\chi$ be $\phi$, we derive that $\phi\in t$.

`Only if': assume that $W\phi\notin s$, to show that $\M^c,s\nvDash W\phi$. If $W\psi\in s$ for {\em no} $\psi$, then we are done, since otherwise we would have $\M^c,s\vDash\phi$ and $\M^c,s\nvDash\phi$. Now we consider the case that $W\psi\in s$ for {\em some} $\psi$. For this, suppose that $\neg\phi\in s$, by IH and Lindenbaum's Lemma, we only need to show that $\{\chi\mid W(\chi\land\psi)\in s\}\cup\{\neg\phi\}$ (denoted $\Gamma$) is consistent. 

%If $\{\chi\mid W(\chi\land\psi)\in s\}$ is empty, then by supposition, $\Gamma=\{\neg\phi\}$ is consistent. In what follows, we only consider the case that $\{\chi\mid W(\chi\land\psi)\in s\}$ is nonempty. 
Since $W\psi\in s$, $\{\chi\mid W(\chi\land\psi)\in s\}$ is nonempty. If $\Gamma$ is not consistent, then there are $\chi_1,\dots,\chi_n$ such that $W(\chi_i\land\psi)\in s$ for all $i=1,\dots,n$ and 
$$\vdash \chi_1\land\cdots\land\chi_n\to\phi.$$
By Prop.~\ref{prop.admissiblerule},
$$\vdash W(\chi_1\land\psi)\land \cdots\land W(\chi_n\land\psi)\land \neg\phi\to W\phi.$$
As $W(\chi_i\land\psi)\in s$ for all $i=1,\dots,n$ and $\neg\phi\in s$, we infer that $W\phi\in s$, which contradicts the assumption, as desired.
\end{proof}

It is now routine to show the following.
\begin{theorem}\label{thm.comp-k}
    ${\bf K^W}$ is sound and strongly complete with resepct to the class of all frames.
\end{theorem}

\subsubsection{Serial Logic}

Let ${\bf KD^W}$ denote ${\bf K^W}+\text{AD}$, where AD is $\neg W\bot$.
\begin{theorem}\label{thm.comp-d}
${\bf KD^W}$ is sound and strongly complete with respect to the class of serial frames.
\end{theorem}

\begin{proof}
For soundness, by Thm.~\ref{thm.comp-k}, it remains only to show the validity of the axiom AD.

If there is a serial model $\M=\lr{S,T,V}$ and a state $s\in S$ such that $\M,s\vDash W\bot$. Then $\M,s\vDash \top$ and for all $t$, if $sRt$ then $\M,t\vDash \bot$. This is impossible since $R$ is serial. Hence $\neg W\bot$ is valid over the class of serial models.

For completeness, define $\M^c$ w.r.t. ${\bf KD^W}$ as in Def.~\ref{def.cm}. By Thm.~\ref{thm.comp-k}, it suffices to prove that $R^c$ is serial. For this, assume that $s\in S^c$. We consider two cases. If there is no $\psi$ such that $W\psi\in s$, then by definition of $R^c$, $sR^cs$. The remainder is the case that there is some $\psi$ such that $W\psi\in s$. In this case, the set $\{\phi\mid W(\phi\land\psi)\in s\}$ is nonempty. By definition of $R^c$ and Lindenbaum's Lemma, we only need to show that this set is consistent.

If not, then there are $\phi_1,\dots,\phi_n$ such that $W(\phi_i\land\psi)\in s$ for all $i=1,\dots,n$, and
$$\vdash \phi_1\land\cdots\land\phi_n\to\bot.$$
By Prop.~\ref{prop.admissiblerule},
$$\vdash W(\phi_1\land\psi)\land \cdots\land W(\phi_n\land\psi)\land\neg\bot\to W\bot.$$
Since $W(\phi_i\land\psi)\in s$ for all $i=1,\dots,n$, and $\neg \bot\in s$, we infer that $W\bot\in s$, which contradicts the fact that $\vdash \neg W\bot$.
\end{proof}

\subsubsection{Reflexive Logic}

Let ${\bf T^W}$ denote ${\bf K^W}+\neg W\phi$.
\begin{theorem}
    ${\bf T^W}$ is sound and strongly complete with respect to the class of reflexive frames.
\end{theorem}

\begin{proof}
    For soundness, by Thm.~\ref{thm.comp-k}, it remains only to show the validity of $\neg W\phi$, which can be obtained from Fact~\ref{fact.imp}.

    For completeness, define $\M^c$ w.r.t. ${\bf T^W}$ as in Def.~\ref{def.cm}. By Thm.~\ref{thm.comp-k}, it suffices to show that $R^c$ is reflexive. For this, let $s\in S^c$. Since $\vdash \neg W\phi$, there is no $\psi$ such that $W\psi\in s$. By definition of $R^c$, we derive that $sR^cs$, as desired.
\end{proof}

\subsubsection{Transitive Logic}\label{sec.transitivelogic}

Let ${\bf K4^W}$ denote the extension of ${\bf K^W}$ with the following axiom:
\[
\begin{array}{ll}
     \text{A4} & W\psi\land W(\phi\land\psi)\to W((W\chi\to W(\phi\land \chi))\land\psi).\\
    %\text{A4-1} & W\psi\land W(\phi\land\psi)\to W(W(\phi\land\psi)\land\psi) \\
    %\text{A4-2} & W\psi\land W(\phi\land\psi)\to W(W\psi\land\psi)\\ 
\end{array}
\]

As mentioned in the introduction, a difficult thing in axiomatizing $\mathcal{L}(W)$ over transitive frames is how to find the desired core axiom. Here the axiom A4 is obtained from the modal axiom 4 (i.e. $\Box\phi\to\Box\Box\phi$) via a translation induced by Almost Definability Schema.
\[
\begin{array}{llr}
   & W\psi\to (\Box\phi\to \Box(W\chi\to\Box\phi)) & (1) \\
  \Longleftrightarrow & W\psi\to (W(\phi\land\psi)\to W((W\chi\to\Box\phi)\land\psi)  & (2)\\
  \Longleftrightarrow & W\psi\to (W(\phi\land\psi)\to W((W\chi\to W(\phi\land\chi))\land\psi)) & (3)\\
  \Longleftrightarrow & W\psi\land W(\phi\land\psi)\to W((W\chi\to W(\phi\land\chi))\land\psi) & (4)\\
\end{array}
\]
We write $W\psi\to (\Box\phi\to \Box(W\chi\to\Box\phi))$ rather than $\Box\phi\to\Box\Box\phi$, since $\Box$ is definable in terms of $W$ under the condition $W\psi$ for some $\psi$. Note that every transformation is equivalent. The above transitions from (1) to (2) and from (2) to (3) follow from Prop.~\ref{prop.almost-definability-schema}. By using propositional calculus (axiom A0), we then obtain the axiom (4), that is, A4.

%A4-2 is equivalent to $W\psi\land W(\phi\land\psi)\to W\bot$)

\begin{proposition}\label{prop.validity-transitiveaxiom}
A4 is valid on the class of transitive frames.
\end{proposition}

\begin{proof}
    %By soundness of ${\bf K^W}$, it remains only to show the validity of axiom A4. 
    Let $\M=\lr{S,R,V}$ be a transitive model and $s\in S$. Suppose, for reductio, that $\M,s\vDash W\psi\land W(\phi\land\psi)$ but $\M,s\nvDash W((W\chi\to W(\phi\land \chi))\land\psi)$. From $\M,s\vDash W\psi$ it follows that $R(s)\vDash\psi$ and $\M,s\nvDash\psi$, thus $\M,s\nvDash (W\chi\to W(\phi\land \chi))\land\psi$. This plus $\M,s\nvDash W((W\chi\to W(\phi\land \chi))\land\psi)$ implies that $R(s)\nvDash (W\chi\to W(\phi\land \chi))\land\psi$, that is, there exists $t$ such that $sRt$ and $\M,t\nvDash (W\chi\to W(\phi\land\chi))\land\psi$. Since $R(s)\vDash\psi$, we infer that $\M,t\vDash \psi$, thus $\M,t\nvDash W\chi\to W(\phi\land\chi)$, namely $t\vDash W\chi$ and $t\nvDash W(\phi\land\chi)$. This entails that $t\nvDash\chi$ and thus $t\nvDash \phi\land\chi$, hence there is a $u$ such that $tRu$ and $\M,u\nvDash\phi\land\chi$. However, we have also $\M,u\vDash\phi\land\chi$: $\M,u\vDash\chi$ follows from $t\vDash W\chi$ and $tRu$, whereas $\M,u\vDash\phi$ is due to the fact that $sRu$ (this is because $sRt$ and $tRu$ and $R$ is transitive) and $s\vDash W(\phi\land\psi)$. A contradiction.
\end{proof}

With the previous preparation in hand, we can show the following.
\begin{theorem}\label{thm.comp-k4}
    ${\bf K4^W}$ is sound and strongly complete with respect to the class of transitive frames.
\end{theorem}

\begin{proof}
The soundness follows immediately from Thm.~\ref{thm.comp-k} and Prop.~\ref{prop.validity-transitiveaxiom}.

For completeness, define $\M^c$ w.r.t. ${\bf K4^W}$ as in Def.~\ref{def.cm}. By Thm.~\ref{thm.comp-k}, it suffices to show that $R^c$ is transitive. Let $s,t,u\in S^c$. Suppose that $sR^ct$ and $tR^cu$, to prove that $sR^cu$. We consider the following cases.
    \begin{itemize}
        \item $W\psi\in s$ for no $\psi$. In this case, by definition of $R^c$, $s=t$. Then $sR^cu$.
        \item $W\psi\in t$ for no $\psi$. Similar to the first case, we can show that $sR^cu$.
        \item $W\psi\in s$ for some $\psi$ and $W\psi'\in t$ for some $\psi'$. In this case, assume for all $\phi$ that $W(\phi\land\psi)\in s$, to show that $\phi\in u$. Using axiom A4, we derive that $W((W\psi'\to W(\phi\land\psi'))\land\psi)\in s$. By $sR^ct$ and definition of $R^c$, we infer that $W\psi'\to W(\phi\land\psi')\in t$, thus $W(\phi\land\psi')\in t$. Now using $tR^cu$ and definition of $R^c$, we conclude that $\phi\in u$, as desired.
    \end{itemize}
\end{proof}

We have thus solved an open problem raised in~\cite{Steinsvold:falsebelief}.
By Thm.~\ref{thm.comp-d} and Thm.~\ref{thm.comp-k4}, we have the following.
\begin{theorem}
${\bf KD4^W}$ is sound and strongly complete with respect to the class of $\mathcal{D}4$-frames.
\end{theorem}

This also answers another open problem raised in~\cite[Sect.~5]{Steinsvold:falsebelief}.

\subsubsection{Euclidean Logic}

Let ${\bf K5^W}$ denote the extension of ${\bf K^W}$ with the following axiom:
\[
\begin{array}{ll}
   \text{A5}  & W\psi\land\neg W(\phi\land\psi)\to W((W\chi\to \neg W(\phi\land\chi))\land\psi) \\
\end{array}
\]

Again, the axiom A5 is obtained from the modal axiom 5 (i.e. $\neg\Box\phi\to\Box\neg\Box\phi$) via a translation induced by Almost Definability Schema.
\[
\begin{array}{llr}
   & W\psi\to (\neg\Box\phi\to \Box(W\chi\to\neg\Box\phi)) & (1') \\
  \Longleftrightarrow & W\psi\to (\neg W(\phi\land\psi)\to W((W\chi\to\neg\Box\phi)\land\psi)  & (2')\\
  \Longleftrightarrow & W\psi\to (\neg W(\phi\land\psi)\to W((W\chi\to \neg W(\phi\land\chi))\land\psi)) & (3')\\
  \Longleftrightarrow & W\psi\land \neg W(\phi\land\psi)\to W((W\chi\to \neg W(\phi\land\chi))\land\psi) & (4')\\
\end{array}
\]

Here, we write $W\psi\to (\neg\Box\phi\to \Box(W\chi\to\neg\Box\phi))$ instead of $\neg\Box\phi\to\Box\neg\Box\phi$, since $\Box$ is definable in terms of $W$ provided that $W\psi$ for some $\psi$. Again, every transformation is equivalent. The above transitions from $(1')$ to $(2')$ and from $(2')$ to $(3')$ follow from Prop.~\ref{prop.almost-definability-schema}. Then by using axiom A0, we get the axiom $(4')$, that is, A5.

Different from our ${\bf K5^W}$, the Euclidean system in~\cite{Steinsvold:falsebelief}, denoted ${\bf S^W}\oplus \text{A}^Q$ there, is defined as the extension of ${\bf S^W}$ (that is, our ${\bf K^W}$) with $\text{A}^Q$, where $\text{A}^Q$ is $W\phi\to W(\neg W\psi\land \phi)$. It is shown in~\cite[Thm.~4.15]{Steinsvold:falsebelief} that ${\bf S^W}\oplus \text{A}^Q$ is sound and complete with respect to the class of Euclidean (and transitive) frames. In what follows, we show that $\text{A}^Q$ is provable in ${\bf K5^W}$.
\begin{proposition}
    $\text{A}^Q$ is provable in ${\bf K5^W}$. 
\end{proposition}

\begin{proof}
    We have the following proof sequence in ${\bf K5^W}$.
    \[
    \begin{array}{lll}
        (1) & W\psi\land\neg W(\neg W\psi\land\phi)\to W((W\psi\to \neg W(\neg W\psi\land\psi))\land\phi) & \text{A5} \\
        (2) & \psi\to\neg W\psi & \text{A1}\\
        (3) & (\neg W\psi\land\psi)\lra \psi & (2),\text{A0}\\
        (4) & W(\neg W\psi\land\psi)\lra W\psi & (3),\text{REW} \\
        (5) & (W\psi\to \neg W(\neg W\psi\land\psi))\lra (W\psi\to\neg W\psi) & (4),\text{A0}\\
        (6) & (W\psi\to \neg W(\neg W\psi\land\psi))\lra \neg W\psi & (5),\text{A0}\\
        (7) & ((W\psi\to \neg W(\neg W\psi\land\psi))\land \phi)\lra  (\neg W\psi\land\phi) & (6),\text{A0}\\
        (8) & W((W\psi\to \neg W(\neg W\psi\land\psi))\land \phi)\lra W (\neg W\psi\land\phi) & (7),\text{REW}\\
        (9) & W\psi\land\neg W(\neg W\psi\land\phi)\to W (\neg W\psi\land\phi) & (1),(8)\\
        (10) & W\psi\to W (\neg W\psi\land\phi) & (9),\text{A0}\\
    \end{array}
    \]
\end{proof}

Below, we will demonstrate that our axiom A5 is valid over the class of Euclidean frames.

\begin{proposition}
    A5 is valid on the class of Euclidean frames.
\end{proposition}

\begin{proof}
    Let $\M=\lr{S,R,V}$ be an Euclidean model and $s\in S$.

    Suppose, for reductio, that $\M,s\vDash W\psi\land\neg W(\phi\land\psi)$ but $\M,s\nvDash W((W\chi\to \neg W(\phi\land\chi))\land\psi)$. Then $\M,s\nvDash\psi$, thus $\M,s\nvDash\phi\land\psi$ and $\M,s\nvDash (W\chi\to \neg W(\phi\land\chi))\land\psi$. It follows that there exists $t$ such that $sRt$ and $\M,t\nvDash \phi\land\psi$, and there exists $u$ such that $sRu$ and $\M,u\nvDash (W\chi\to \neg W(\phi\land\chi))\land\psi$. Using $s\vDash W\psi$ again, we derive that $t\vDash\psi$ and $u\vDash \psi$, and then $t\nvDash \phi$ and $u\nvDash W\chi\to \neg W(\phi\land\chi)$, that is, $u\vDash W\chi\land W(\phi\land\chi)$. By $sRu$ and $sRt$ and the Euclideanness of $R$, we have $uRt$. Then it follows from $u\vDash W(\phi\land\chi)$ that $t\vDash\phi$, as desired.
\end{proof}

\begin{proposition}\label{prop.sound-k5}
    ${\bf K5^W}$ is sound with respect to the class of Euclidean frames.
\end{proposition}

Now we demonstrate the completeness of ${\bf K5^W}$ over Euclidean frames. Our proof is different from that used in~\cite[Thm.~4.15]{Steinsvold:falsebelief}. The proof is nontrivial. This is because the canonical model is secondarily reflexive (defined later), not Euclidean. Thus we need to transform the secondarily reflexive model into an Euclidean model, and the truth values of $\mathcal{L}(W)$-formulas have to be preserved during the transformation. This is our strategy. To begin with, we need a notion of secondary reflexivity.

We say that a model $\M=\lr{S,R,V}$ is {\em secondarily reflexive}, if for all $s,t\in S$, $sRt$ implies $tRt$. We have the following general result, which will be used in the proof of the completeness of ${\bf K5^W}$ (Thm.~\ref{thm.comp-k5}).

\begin{proposition}\label{prop.equivalentmodel}
For every secondarily reflexive model $\M=\lr{S,R,V}$, there exists an Euclidean model $\M'=\lr{S,R',V}$ such that for all $s\in S$, for all $\phi\in\mathcal{L}(W)$, $\M,s\vDash\phi$ iff $\M',s\vDash\phi$.
%Every secondarily reflexive model is $\mathcal{L}(W)$-equivalent to an Euclidean model. 
\end{proposition}

\begin{proof}
    Let $\M=\lr{S,R,V}$ be a secondarily reflexive model. Construct a model $\M'=\lr{S,R',V}$ such that $R'=R\cup\{(y,z)\mid xRy~\text{and}~x'Rz\text{~for some~}x,x'\in S\}$.
    
    First, $\M'$ is Euclidean. Let $s,t,u\in S$ such that $sR't$ and $sR'u$. The goal is to show $tR'u$. By definition of $R'$, we consider the following cases.
    \begin{itemize}
        \item $sRt$ and $sRu$. Then $tR'u$.
        \item ${\sim}sRt$ and $sRu$. Then $xRs$ and $x'Rt$ for some $x,x'\in S$. Then $tR'u$.
        \item $sRt$ and ${\sim}sRu$. Similar to the second case, we can show that $tR'u$.
        \item ${\sim}sRt$ and ${\sim}sRu$. Then $xRs$ and $x'Rt$ for some $x,x'\in S$, and $yRs$ and $y'Ru$ for some $y,y'\in S$. Then $tR'u$.
    \end{itemize}

    It remains only to show that for all $s\in S$, for all $\phi\in\mathcal{L}(W)$, we have
    $$\M,s\vDash\phi\text{~iff~}\M',s\vDash\phi.$$
    We proceed by induction on $\phi$. The nontrivial case is $W\phi$.

    Suppose that $\M,s\nvDash W\phi$, to show that $\M',s\nvDash W\phi$. By supposition, either $\M,s\vDash\phi$ or for some $t$ such that $sRt$ we have $\M,t\nvDash \phi$. By induction hypothesis and $R\subseteq R'$, $\M',s\vDash\phi$ or for some $t$ such that $sR't$ and $\M',t\nvDash \phi$. Thus $\M',s\nvDash W\phi$.

    Conversely, assume that $\M',s\nvDash W\phi$, to prove that $\M,s\nvDash W\phi$. By assumption, either $\M',s\vDash\phi$ or for some $t$ such that $sR't$ we have $\M',t\nvDash\phi$. If the first case holds, by induction hypothesis, we derive that $\M,s\vDash\phi$, thus $\M,s\nvDash W\phi$. If the second case holds, according to the definition of $R'$, we consider the following two cases.%Otherwise, that is, $\M',s\nvDash\phi$, then for some $t$ such that $sT't$ and $\M',t\nvDash\phi$. According to the definition of $T'$, we consider the following two cases.
    \begin{itemize}
        \item $sRt$. By $\M',t\nvDash\phi$ and induction hypothesis, $\M,t\nvDash\phi$. Thus $\M,s\nvDash W\phi$.
        \item $xRs$ and $x'Rt$ for some $x,x'\in S$. Since $xRs$ and $\M$ is secondarily reflexive, it follows that $sRs$. Then using Fact~\ref{fact.imp}, we conclude that $\M,s\nvDash W\phi$.%By $\M',s\nvDash\phi$ and induction hypothesis, $\M,s\nvDash\phi$. We have thus shown that $sTs$ and $\M,s\nvDash\phi$, therefore $\M,s\nvDash W\phi$.
    \end{itemize}
\end{proof}

The reader may ask if the above statement can be extended to the case of transitivity and serial. That is, do we have the following: Every (serial,) transitive and secondarily reflexive model is $\mathcal{L}(W)$-equivalent to a (serial,) transitive and Euclidean model? We do not the answer. As we check, the construction $\M'$ in the proof of Prop.~\ref{prop.equivalentmodel} does not preserve transitivity. We will come back to this issue.

\weg{Similarly, we can show the following.
\begin{proposition}\label{prop.equi-4}
    Every transitive and secondarily reflexive model is $\mathcal{L}(W)$-equivalent to a transitive and Euclidean model.
\end{proposition}

\begin{proposition}\label{prop.d4}
    Every serial, transitive and secondarily reflexive model is $\mathcal{L}(W)$-equivalent to a serial, transitive and Euclidean model.
\end{proposition}}

%However, $\M^c$ is not Euclidean. Rather, it is a secondarily reflexive model.

\begin{theorem}\label{thm.comp-k5}
${\bf K5^W}$ is sound and strongly complete with respect to the class of Euclidean frames.
\end{theorem}

\begin{proof}
    By Prop.~\ref{prop.sound-k5}, it suffices to show the completeness of ${\bf K5^W}$. For this, define $\M^c$ w.r.t. ${\bf K5^W}$ as in Def.~\ref{def.cm}. Firstly, we show that $\M^c$ is secondarily reflexive,
    that is, the following holds:
    \begin{center}
    (*) for all $s,t\in S^c$, if $sR^ct$ then $tR^ct$.
    \end{center}
     
    Let $s,t\in S^c$. Suppose that $sR^ct$, to show that $tR^ct$. According to the definition of $R^c$, we consider the following cases.
    \begin{itemize}
        \item There is no $\psi$ such that $W\psi\in s$. Then $s=t$. Thus $tR^ct$.
        \item There is no $\psi'$ such that $W\psi'\in t$. Then as $t=t$, we also have $tT^ct$.
        \item $W\psi\in s$ and $W\psi'\in t$ for some $\psi$ and $\psi'$. If it fails that $tR^ct$, according to the definition of $R^c$, it follows that for some $\phi$, $W(\phi\land\psi')\in t$ and $\phi\notin t$. As $sR^ct$, we must have $W(\phi\land\psi)\notin s$, thus $\neg W(\phi\land\psi)\in s$. Using axiom A5, we infer that $W((W\psi'\to\neg W(\phi\land\psi'))\land\psi)\in s$. Using $sR^ct$ again, we derive that $W\psi'\to \neg W(\phi\land\psi')\in t$, thus $\neg W(\phi\land\psi')\in t$, which is contrary to $W(\phi\land\psi')\in t$ and the consistency of $t$.
    \end{itemize}
    We have thus shown (*). This implies that $\M^c$ is a secondarily reflexive model. That is to say, every consistent set is satisfiable in a secondarily reflexive model.

    %In what follows, we show that $T^c$ is Euclidean. For this, given any $s,t,u\in S^c$. Assume that $sT^ct$ and $sT^cu$, to prove that $tT^cu$. By assumption and (*), $tT^ct$ 
    Now by Prop.~\ref{prop.equivalentmodel}, we obtain that every consistent set is satisfiable in an Euclidean model, as desired.
\end{proof}

It may be worth remarking that axiom A4 is provable in ${\bf K5^W}$, because it is valid on the class of Euclidean frames.
\begin{proposition}
    A4 is valid on the class of Euclidean frames.
\end{proposition}

\begin{proof}
    Let $\M=\lr{S,R,V}$ be an Euclidean model and $s\in S$. Suppose, for a contradiction, that $\M,s\vDash W\psi\land W(\phi\land\psi)$ and $\M,s\nvDash W((W\chi\to W(\phi\land \chi))\land\psi)$. Then $\M,s\nvDash\psi$, thus $s\nvDash (W\chi\to W(\phi\land \chi))\land\psi$. It then follows that there exists $t$ such that $sRt$ and $\M,t\nvDash (W\chi\to W(\phi\land \chi))\land\psi$. Moreover, as $s\vDash W\psi$, $t\vDash\psi$, thus $t\nvDash W\chi\to W(\phi\land\chi)$, and hence $t\vDash W\chi$. However, since $sRt$ and $R$ is Euclidean, $tRt$. By Fact~\ref{fact.imp}, we should have also $t\nvDash W\chi$. A contradiction.
\end{proof}

So ${\bf K4^W}\subseteq {\bf K5^W}$. Note that in the above proof, $s\vDash W(\phi\land\psi)$ is not needed. This means that a stronger version of A4, that is, $W\psi\to W((W\chi\to W(\phi\land \chi))\land\psi)$ is valid over the class of Euclidean frames, thus provable in ${\bf K5^W}$. In contrast, this formula is not valid over the class of transitive frames (as one may verify), thus not provable in ${\bf K4^W}$. This establishes that ${\bf K4^W}\subset {\bf K5^W}$.

Moreover, ${\bf K45^W}={\bf K5^W}$, where ${\bf K45^W}$ is the extension of ${\bf K5^W}$ with the axiom A4. As a consequence, we have another completeness result.
\begin{theorem}
    ${\bf K45^W}$ is sound and strongly complete with respect to the class of Euclidean frames.
\end{theorem}

%Besides, we have the following.

\begin{theorem}\label{thm.comp-k5-45}
    ${\bf K5^W}(={\bf K45^W})$ is sound and strongly complete with respect to the class of transitive and Euclidean frames.
\end{theorem}

\begin{proof}
    The soundness is direct from Thm.~\ref{thm.comp-k5}.

    For the completeness, define $\M^c$ w.r.t. ${\bf K5^W}$ as in Def.~\ref{def.cm}. We have shown that $\M^c$ is transitive (Thm.~\ref{thm.comp-k4}) and secondarily reflexive (Thm.~\ref{thm.comp-k5}). This entails that every consistent set, say $\Gamma$, is satisfiable in a transitive and secondarily reflexive model, say $(\M,s)$. Let $\M'=\lr{S,R,V}$ is the submodel of $\M$ generated by $s$. By the generated submodel theorem for standard modal logic $\mathcal{L}(\Box)$, we have $\M',s\vDash \Gamma$. Now construct a new model $\mathcal{N}=\lr{S,R',V}$ such that $R'=R\cup(Z(s)\times Z(s))$, where $Z(s)=\{x\mid sRx\}$. We can see that $\mathcal{N}$ is transitive and Euclidean. 
    
    It remains only to show that for all $x\in S$, for all $\phi\in\mathcal{L}(W)$, $\M',x\vDash\phi$ iff $\mathcal{N},x\vDash\phi$. We proceed by induction on $\phi$. The only nontrivial case is $W\phi$.

    Suppose that $\M',x\nvDash W\phi$. Then $\M',x\vDash\phi$ or for some $y$ such that $xRy$ and $\M',y\nvDash\phi$. By induction hypothesis and $R\subseteq R'$, $\mathcal{N},x\vDash\phi$ or for some $y$ such that $xR'y$ and $\M',y\nvDash\phi$. Thus $\mathcal{N},x\nvDash W\phi$.

    Conversely, assume that $\mathcal{N},x\nvDash W\phi$. Then $\mathcal{N},x\vDash\phi$ or for some $y$ such that $xT'y$ and $\mathcal{N},y\nvDash\phi$. If $\mathcal{N},x\vDash\phi$, by induction hypothesis, $\mathcal{M}',x\vDash\phi$, thus $\M',x\nvDash W\phi$. If for some $y$ such that $xR'y$ and $\mathcal{N},y\nvDash\phi$, according to the defintion of $R'$, we consider two cases.
    \begin{itemize}
        \item $xRy$. Then by induction hypothesis and $\mathcal{N},y\nvDash\phi$, we have $\M',y\nvDash\phi$, and then $\M',x\nvDash W\phi$.
        \item $(x,y)\in Z(s)\times Z(s)$. Then $x\in Z(s)$, that is, $sRx$. Since $R$ is secondarily reflexive (note that the property of secondary reflexivity is preserved under generated submodels), it follows that $xRx$. By Fact~\ref{fact.imp}, $\M',x\nvDash W\phi$.
    \end{itemize}

    Since $\M',s\vDash\Gamma$, we infer that $\mathcal{N},s\vDash\Gamma$. Thus $\Gamma$ is satisfiable in a transitive and Euclidean model, as desired.
\end{proof}

Similarly, we can show the following. Let ${\bf KD5^W}$ is the extension of ${\bf K5^W}$ with the axiom $\neg W\bot$.
\begin{theorem}\label{thm.comp-kd5}
    ${\bf KD5^W}={\bf KD45^W}$ is sound and strongly complete with respect to the class of serial, transitive and Euclidean frames.
\end{theorem}

Going back to the discussion after Prop.~\ref{prop.equivalentmodel}, although the construction $\M'$ in the proof of Prop.~\ref{prop.equivalentmodel} does not preserve transitivity, this property is indeed preserved under generated submodels and the construction of $\mathcal{N}$ in Thm.~\ref{thm.comp-k5-45} and also Thm.~\ref{thm.comp-kd5}. 
%Note that in the semantics of $\I^R_a\phi$, although $\I^R_a$ is `defined' with $\I_a$, $\I_a\phi$ is {\em not} a subformula of $\I^R_a\phi$. This brings about a technical difficulty in the completeness proof, as we shall see below. 

%\subsection{Symmetric Logic}

In a similar vein, by translating axiom $B$ (viz. $\neg\phi\to \Box\neg\Box\phi$) via the translation induced by Almost Definability Schema, we can obtain an axiom $W\psi\land\neg\phi\to W((W\chi\to\neg W(\phi\land\chi))\land\psi)$ of $\mathcal{L}(W)$ (denoted AB) over symmetric frames. One may verify that AB is valid over the class of symmetric frames.

\section{Radical Ignorance}\label{sec.radicalignorance}

\begin{definition}[Language] The language of the logic of radical ignorance, denoted $\mathcal{L}(I_R)$, is defined recursively as follows:
%$$\phi::=p\mid \neg\phi\mid (\phi\land\phi)\mid \I\phi\mid \I_R\phi.$$
$$\phi::=p\mid \neg\phi\mid (\phi\land\phi)\mid \I_R\phi.$$
\end{definition}

Intuitively, $\I^R\phi$ is read ``one is {\em Rumsfeld ignorant of} $\phi$''. Other connectives are defined as usual.

%Intuitively, $\I_a\phi$ is read ``the agent $a$ is {\em (first-order) ignorant whether} $\phi$'', and in the context of contingency logic, $\I_a\phi$ is read ``$\phi$ is contingent for $i$''~\cite{Fanetal:2015}; $\I^R_a\phi$ is read ``the agent $a$ is {\em Rumsfeld ignorant of} $\phi$''. Other connectives are defined as usual; in particular, $\Kw_a\phi$, read ``$a$ knows whether $\phi$'', is defined as $\neg\I_a\phi$.

%The above language is interpreted over (relational) models.
\weg{\begin{definition}[Structure]
A tuple $\M=\lr{S,R,T,V}$ is said to be a {\em model}, if .
%A {\em model} is a tuple $\M=\lr{S,R,V}$, where $S$ is a nonempty set of states (or points), for every $a\in \Ag$, $R_a$ is a binary relation over $S$, called `accessibility relation', and $V$ is a valuation function. A {\em pointed model} is a pair of a model with a point in it. A {\em frame} is a model without a valuation. Model $\M$ is said to be a $45$-model (resp. a $\mathcal{KD}45$-model, an $\mathcal{S}4$-model, an $\mathcal{S}5$-model), if for each $a\in\Ag$, $R_a$ is transitive and Euclidean (resp. serial and transitive and Euclidean, reflexive and transitive, reflexive and Euclidean). A $45$-frame and the like are defined similarly.
\end{definition}}

The notions of models and frames are defined as in Def.~\ref{def.structure}, and the semantics of $\mathcal{L}(I_R)$ is defined as in Def.~\ref{def.semantics}, except that
\[
\begin{array}{lll}
\M,s\vDash\I_R\phi & \text{iff}& \text{{\em either} }(R(s)\vDash\phi\text{ and } \M,s\nvDash\phi)\\
    &&\text{\em or }(R(s)\vDash\neg\phi\text{ and }\M,s\vDash\phi). \\
\end{array}
\]
%In what follows, we will use $45$, $\mathcal{KD}45$, $\mathcal{S}4$, $\mathcal{S}5$, respectively, to denote the class of frames possessing the 

\weg{\begin{definition}[Semantics]
Given a model $\M=\lr{S,T,V}$ and a state $s\in S$, the semantics of $\mathcal{L}(I_R)$ is defined recursively as follows.
\[
\begin{array}{|lll|}
\hline
    \M,s\vDash p & \text{iff} & s\in V(p)\\
    \M,s\vDash\neg\phi & \text{iff} & \M,s\nvDash\phi\\
    \M,s\vDash\phi\land\psi&\text{iff} & \M,s\vDash\phi\text{ and }\M,s\vDash\psi\\
%    \M,s\vDash\I\phi &\text{iff}& \text{for some }t\text{ such that }sRt\text{ and }\M,t\vDash\phi,\text{ and }\\
 %   &&\text{for some }u\text{ such that }sRu\text{ and }\M,u\nvDash\phi\\    
 %\M,s\vDash\I_R\phi & \text{iff}& \M,s\vDash\I\phi\text{ and {\em either} }(T(s)\vDash\phi\text{ and } \M,s\nvDash\phi)\\
%    &&\text{\em or }(T(s)\vDash\neg\phi\text{ and }\M,s\vDash\phi) \\
\M,s\vDash\I_R\phi & \text{iff}& \text{{\em either} }(T(s)\vDash\phi\text{ and } \M,s\nvDash\phi)\\
    &&\text{\em or }(T(s)\vDash\neg\phi\text{ and }\M,s\vDash\phi) \\
    \hline
\end{array}
\]
\end{definition}
}

Note that the semantics of $\I_R\phi$ is equivalent to the following one.
\[
\begin{array}{lll}
\M,s\vDash\I_R\phi & \text{iff}  &(R(s)\vDash\phi\text{ or }\M,s\vDash\phi)\text{ and}\\
&&(R(s)\vDash\neg\phi\text{ or }\M,s\vDash\neg\phi)  \\ 
\end{array}
\]
\weg{\[
\begin{array}{lll}
\M,s\vDash\I_R\phi & \text{iff}  &\M,s\vDash\I\phi\text{ and }(T(s)\vDash\phi\text{ or }\M,s\vDash\phi)\text{ and}\\
&&(T(s)\vDash\neg\phi\text{ or }\M,s\vDash\neg\phi)  \\ 
\end{array}
\]}
We will use the two semantics of $\I_R$ interchangeably.
\weg{\[
\begin{array}{lll}
\M,s\vDash\I_R\phi & \text{iff}& \text{{\em either} }(T(s)\vDash\phi\text{ and } \M,s\nvDash\phi)\\
    &&\text{\em or }(T(s)\vDash\neg\phi\text{ and }\M,s\vDash\phi) \\
\end{array}
\]

\[
\begin{array}{lll}
\M,s\vDash\I_R\phi & \text{iff}  &(T(s)\vDash\phi\text{ or }\M,s\vDash\phi)\text{ and}\\
&&(T(s)\vDash\neg\phi\text{ or }\M,s\vDash\neg\phi)  \\ 
\end{array}
\]}

Recall that the operator $W$ of false belief is interpreted as follows:
\[
\begin{array}{lll}
    \M,s\vDash W\phi & \text{iff} & R(s)\vDash\phi\text{ and }\M,s\nvDash\phi.\\ 
\end{array}
\]

One may check that the operators of radical ignorance and of false belief are interdefined with each other, as $\vDash I_R\phi\lra (W\phi\vee W\neg\phi)$ and $\vDash W\phi\lra (I_R\phi\land\neg\phi)$. This may indicate that one can translate the results about false belief into those about radical ignorance via the translation induced by the interdefinability of the operators. Unfortunately, this holds for all but proof systems. %This enables us to translate the results about false belief to those about radical ignorance. 
Here we illustrate this with the axiomatizations of the logic of radical ignorance over all frames and over serial and transitive frames. %So the relationship of $W$ and $I_R$ is kind of similar to that of the `knowing whether' operator $Kw$ and the `knowing that' operator $K$.

\subsection{Minimal logic}\label{sec.minimallogic}

The minimal logic of $\mathcal{L}(I_R)$, denoted ${\bf K^{RI}}$, consists of the following axioms and inference rules:
\[
\begin{array}{ll}
    \text{A0} & \text{all instances of propositional tautologies} \\
    \text{RI-Equ} & \I_R\phi\lra \I_R\neg\phi\\
    \text{RI-Con} & \I_R\phi\land\neg\phi\land\I_R\psi\land\neg\psi\to \I_R(\phi\land\psi)\\
    \text{MP} & \dfrac{\phi,\phi\to\psi}{\psi}\\
    \text{RI-R} & \dfrac{\phi\to\psi}{\I_R\phi\land\neg\phi\to\I_R\psi\vee\psi}\\
\end{array}
\]

Note that the above axioms and inference rules can be obtained from those of ${\bf K^W}$ by a translation induced by the definability of $W$ in terms of $I_R$, that is $\vDash W\phi\lra (I_R\phi\land\neg\phi)$, except for axiom RI-Equ. Actually, the translation only gives us an incomplete proof system, since the axiom RI-Equ is valid, but not provable in the translated system. To see this, consider an auxiliary semantics in which all
formulas of the form $I_R\phi$ are interpreted as $\phi$. Under this semantics, the translated system are sound, but RI-Equ is not valid.

\begin{proposition}\label{prop.admissible-kri}
    The following rule is admissible in ${\bf K^{RI}}$:
    $$\dfrac{\phi\to\psi}{I_R(\phi\land\chi)\land\neg(\phi\land\chi)\to I_R\psi\vee\psi}.$$
\end{proposition}

\begin{proof}
We have the following proof sequence in ${\bf K^{RI}}$.
\[
\begin{array}{lll}
   (1) & \phi\to\psi & \text{premise} \\
   (2) & I_R\phi\land\neg\phi\to I_R\psi\vee\psi & (1),\text{RI-R}\\
   (3) & I_R\phi\to I_R\psi\vee\psi\vee\phi & (2)\\
   (4) & \phi\land\chi\to\phi & \text{A0}\\
   (5) & I_R(\phi\land\chi)\land\neg (\phi\land\chi)\to I_R\phi\vee\phi & (4),\text{RI-R}\\
   (6) & I_R(\phi\land\chi)\land\neg(\phi\land\chi)\to I_R\psi\vee\psi\vee\phi & (3),(5)\\
   (7) & I_R(\phi\land\chi)\land\neg(\phi\land\chi)\to I_R\psi\vee\psi & (1),(6)\\
\end{array}
\]
\end{proof}

We can generalize the above result to the following.
\begin{proposition}\label{prop.admissiblerule-ri}
    The following rule is admissible: for all $n\in\mathbb{N}$,
    $$\dfrac{\chi_1\land\cdots\land\chi_n\to \phi}{I_R(\chi_1\land\psi)\land \neg(\chi_1\land\psi)\land\cdots\land I_R(\chi_n\land\psi)\land \neg(\chi_n\land\psi)\to I_R\phi\vee \phi}.$$
\end{proposition}

\begin{proof}
By induction on $n\in\mathbb{N}$.
The case $n=0$ is obvious. The case $n=1$ is shown as in Prop.~\ref{prop.admissible-kri}.

Now suppose that the statement holds for the case $n=m$ (IH), to show it also holds for the case $n=m+1$. For this, we have the following proof sequence:
\[
\begin{array}{lll}
   (1) & \chi_1\land\cdots\land\chi_{m+1}\to\phi & \text{premise} \\
   (2) & (\chi_1\land\chi_2)\land\cdots\land\chi_{m+1}\to\phi & (1)\\
   (3) & I_R(\chi_1\land\chi_2\land\psi)\land\neg(\chi_1\land\chi_2\land\psi)\land\cdots\land & \\
   & ~~~~~~I_R(\chi_{m+1}\land\psi)\land\neg (\chi_{m+1}\land\psi)\to I_R\phi\vee\phi & (2),\text{IH}\\
   (4) & I_R(\chi_1\land\psi)\land\neg (\chi_1\land\psi)\land I_R(\chi_2\land\psi)\land\neg (\chi_2\land\psi)\to & \\
   & ~~~~~~I_R(\chi_1\land\chi_2\land\psi)\land \neg(\chi_1\land\chi_2\land\psi) & \text{RI-Con}\\
   (5) & I_R(\chi_1\land\psi)\land\neg(\chi_1\land\psi)\land\cdots\land I_R(\chi_{m+1}\land\psi)& \\ & ~~~~~~\land\neg(\chi_{m+1}\land\psi)\to I_R\phi\vee\phi & (3),(4)\\
\end{array}
\]
\end{proof}

By Def.~\ref{def.cm} and the definability of $W$ in terms of $I_R$, we obtain the canonical model for ${\bf K^{RI}}$ as follows.
\begin{definition}\label{def.cm-kri}
The {\em canonical model for ${\bf K^{RI}}$} is $\M^c=\lr{S^c,R^c, V^c}$, where
\begin{itemize}
    \item $S^c=\{s\mid s\text{ is a maximal consistent set for }\}$
    \item if $I_R\psi\land\neg\psi\in s$ for {\em no} $\psi$, then $sR^ct$ iff $s=t$, and\\
if $I_R\psi\land\neg\psi\in s$ for {\em some} $\psi$, then $sR^ct$ iff for all $\phi$, if $I_R(\phi\land\psi)\land\neg(\phi\land\psi)\in s$, then $\phi\in t$.
\item $V^c(p)=\{s\in S^c\mid p\in s\}$.
\end{itemize}
\end{definition}

\begin{lemma}
    For all $\phi\in\mathcal{L}(I_R)$, for all $s\in S^c$, we have
    $$\M^c,s\vDash\phi\text{~iff~}\phi\in s.$$
\end{lemma}

\begin{proof}
    By induction on $\phi\in\mathcal{L}(I_R)$. The nontrivial case is $I_R\phi$.

    Suppose that $I_R\phi\in s$ (thus $I_R\neg\phi\in s$), to show that $\M^c,s\vDash I_R\phi$. By induction hypothesis, we show that 
\begin{itemize}
    \item[(*)] if $\phi\notin s$, then for all $x\in S^c$ such that $sR^cx$, we have $\phi\in x$, and if $\phi\in s$, then for all $y\in S^c$ such that $sR^cy$, we have $\phi\notin y$.
\end{itemize}
Firstly, we assume that $\phi\notin s$, then $\neg\phi\in s$. By supposition, $I_R\phi\land\neg\phi\in s$. Let $x\in S^c$ such that $sR^cx$. Then according to the definition of $R^c$, we have: for all $\chi$, if $I_R(\chi\land\phi)\land\neg(\chi\land\phi)\in s$, then $\chi\in x$. By letting $\chi$ be $\phi$, we can show that $\phi\in x$. A similar argument applies to the second conjunct of (*).

Conversely, suppose that $I_R\phi\notin s$ (thus $I_R\neg\phi\notin s$), to prove that $\M^c,s\nvDash\I_R\phi$. By induction hypothesis, it suffices to show the following fails:
\begin{itemize}
    \item[(a)] {\em either} $\phi\notin s$ and for all $x\in S^c$ such that $sR^cx$, we have $\phi\in x$, {\em or} $\phi\in s$ and for all $y\in S^c$ such that $sR^cy$, we have $\phi\notin y$.
\end{itemize} 
This amounts to showing the following (a1) and (a2) hold.
\begin{itemize}
    \item[(a1)] if $\phi\notin s$, then for some $x\in S^c$ such that $sR^cx$, we have $\phi\notin x$, {\em and}
    \item[(a2)] if $\phi\in s$, then for some $y\in S^c$ such that $sR^cy$, we have $\phi\in y$.
\end{itemize}
For (a1), assume that $\phi\notin s$. If $I_R\psi\land\neg\psi\in s$ for no $\psi$, then according to the definition of $R^c$, we have $sR^cs$. In this case, $s$ is a desired $x$. If $I_R\psi\land\neg\psi\in s$ for some $\psi$, by definition of $R^c$ and Lindenbaum's Lemma, it remains only to show that the set $\{\chi\mid I_R(\chi\land\psi)\land\neg(\chi\land\psi)\in s\}\cup\{\neg\phi\}$ (denoted $\Gamma$) is consistent.

If $\Gamma$ is {\em not} consistent, then there exist $\chi_1,\ldots,\chi_n$ such that $I_R(\chi_i\land\psi)\land\neg(\chi_i\land\psi)\in s$ for $i=1,\ldots,n$ and
$$\vdash \chi_1\land\cdots\land\chi_n\to\phi.$$
By Prop.~\ref{prop.admissiblerule-ri}, we infer that
$$\vdash I_R(\chi_1\land\psi)\land \neg(\chi_1\land\psi)\land\cdots\land I_R(\chi_n\land\psi)\land \neg(\chi_n\land\psi)\to I_R\phi\vee \phi.$$
As $I_R(\chi_i\land\psi)\land\neg(\chi_i\land\psi)\in s$ for $i=1,\ldots,n$, we derive that $I_R\phi\vee\phi\in s$, which contradicts the supposition and the assumption. Thus we complete the proof of (a1).

Similarly, we can prove (a2), by using $I_R\neg\phi\notin s$ and $\neg\phi\notin s$ instead.
\end{proof}

\begin{theorem}\label{thm.comp-kri}
   ${\bf K^{RI}}$ is sound and strongly complete with respect to the class of all frames.
\end{theorem}

%\subsection{System ${\bf KD4^{RI}}$}
\subsection{Serial and transitive logic}\label{sec.d4-radicalignornce}

In this section, we consider the extension of ${\bf K^{RI}}$ over serial and transitive frames. This is in line with the frames that the framework of~\cite{Fano:2020working} is actually based on, where the doxastic accessibility relation is serial and transitive, see Fn.~\ref{fn.remark} for the remark.

Define ${\bf KD4^{RI}}$ to be the extension of ${\bf K^{RI}}$ with the axiom RI-D ($\neg I^R\bot$) and the following axiom (denoted RI-4):
\[
\begin{array}{l}
  I_R\psi\land\neg\psi\land I_R(\phi\land\psi)\land\neg (\phi\land\psi)\to I_R((I_R\chi\land\neg\chi\to I_R(\phi\land \chi)\\
  ~~~~~~\land\neg (\phi\land\chi))\land\psi)\land
  \neg((I_R\chi\land\neg\chi\to I_R(\phi\land \chi)\land\neg (\phi\land\chi))\land\psi)\\
\end{array}
\]

Again, the above axioms RI-D and RI-4 are obtained from, respectively, axioms AD and A4 via a translation induced by the interdefinability of $W$ in terms of $I_R$.

\begin{theorem}\label{thm.comp-kd4ri}
${\bf KD4^{RI}}$ is sound and strongly complete with respect to the class of serial and transitive frames.    
\end{theorem}

\begin{proof}
For soundness, by Thm.~\ref{thm.comp-kri}, it suffices to show the validity of axioms RI-D and RI-4 over serial and transitive frames. This follows directly from the validity of AD and A4 over the frames under discussion (Thm.~\ref{thm.comp-d} and Prop.~\ref{prop.validity-transitiveaxiom}) with the definability of $W$ in terms of $I_R$.

    For completeness, define $\M^c$ w.r.t. ${\bf KD4^{RI}}$ as in Def.~\ref{def.cm-kri}. By Thm.~\ref{thm.comp-kri}, it remains only to show that $R^c$ is serial and transitive.

    For seriality, suppose that $s\in S^c$. If $I_R\psi\land\neg\psi\in s$ for no $\psi$, then according to the definition of $R^c$, we derive that $sR^cs$. If $I_R\psi\land\neg\psi\in s$ for some $\psi$, by definition of $R^c$ and Lindenbaum's Lemma, it suffices to prove that $\{\chi\mid I_R(\chi\land\psi)\land\neg(\chi\land\psi)\in s\}$ is consistent.

    Since $I_R\psi\land\neg\psi\in s$, the set $\{\chi\mid I_R(\chi\land\psi)\land\neg(\chi\land\psi)\in s\}$ is nonempty. If the set is not consistent, then there are $\chi_1,\ldots,\chi_n$ such that $I_R(\chi_i\land\psi)\land\neg (\chi_i\land\psi) \in s$ for $i=1,\ldots,n$ and
    $$\vdash \chi_1\land\cdots\land\chi_n\to \bot.$$
    By Prop.~\ref{prop.admissiblerule-ri},
    $$\vdash I_R(\chi_1\land\psi)\land\neg(\chi_1\land\psi)\land\cdots\land I_R(\chi_n\land\psi)\land\neg(\chi_n\land\psi)\to I_R\bot\vee \bot.$$
    As $I_R(\chi_i\land\psi)\land\neg(\chi_i\land\psi)\in s$ for $i=1,\ldots,n$, we conclude that $I_R\bot\vee\bot\in s$. As $\bot\notin s$, $I_R\bot\in s$. However, by axiom RI-D, $\neg I_R\bot\in s$. This contradicts the consistency of $s$.

\medskip

    For transitivity, let $s,t,u\in S^c$. Assume that $sR^ct$ and $tR^cu$, to show that $sR^cu$. We consider the following three cases.
    \begin{itemize}
        \item $I_R\psi\land\neg\psi\in s$ for no $\psi$. In this case, by definition of $R^c$ and $sR^ct$, it follows that $s=t$, thus $sR^cu$ by assumption that $tR^cu$.
        \item $I_R\psi\land\neg\psi\in t$ for no $\psi$. In this case, by definition of $R^c$ and $tR^cu$, it follows that $t=u$, thus $sR^cu$ by assumption that $sR^ct$.
        \item $I_R\psi\land\neg\psi\in s$ for some $\psi$, and $I_R\psi'\land\neg\psi'\in t$ for some $\psi'$. In this case, suppose that for any $\phi$ we have $I_R(\phi\land\psi)\land\neg(\phi\land\psi)\in s$, we need to show that $\phi\in u$. Since $I_R\psi\land\neg\psi\in s$, by supposition and axiom RI-4, we derive that $I_R((I_R\psi'\land\neg\psi'\to I_R(\phi\land \psi')\land\neg (\phi\land\psi'))\land\psi)\land  \neg((I_R\psi'\land\neg\psi'\to I_R(\phi\land \psi')\land\neg (\phi\land\psi'))\land\psi)\in s$. As $sR^ct$, it follows that $I_R\psi'\land\neg\psi'\to I_R(\phi\land \psi')\land\neg (\phi\land\psi')\in t$, thus $I_R(\phi\land \psi')\land\neg (\phi\land\psi')\in t$. As $tR^cu$, we conclude that $\phi\in u$, as desired.
    \end{itemize}
\end{proof}

\begin{remark}
    In the introduction, we note that the canonical model in~\cite{Fan:2022sane} does not apply to the transitive logic of reliable belief, thus not apply to the transitive logic of radical ignorance. Recall from~\cite[Def.~6.3]{Fan:2022sane} that the canonical model for the logic of reliable belief is defined such that $sR^ct$ iff for all $\phi$, $\neg \mathcal{R}\phi\land \neg\phi\in s$ implies $\phi\in t$. As observed in the introduction, $I_R$ is equivalent to the negation of $\mathcal{R}$. Accordingly, in the case of radical ignorance, $sR^ct$ iff for all $\phi$, $I_R\phi\land\neg\phi\in s$ implies $\phi\in t$. As the reader check, $R^c$ is not transitive. In contrast, our $R^c$ in Def.~\ref{def.cm-kri} is indeed transitive, as shown in Thm.~\ref{thm.comp-kd4ri}.
\end{remark}

\section{Conclusion and Discussions}

In this paper, we investigated the logics of false belief and radical ignorance. We proposed an almost definability schema, called `Almost Definability Schema', which guides us to find the desired core axioms for the transitive logic and the Euclidean logic of false belief, and (with other considerations) also inspires us to propose a suitable canonical relation in the construction of the canonical model for the minimal logic of false belief. The canonical relation can uniformly handle the completeness proof of various logics of false belief, including the transitive logic, thereby solving an open problem in~\cite{Steinsvold:falsebelief}. We explored the expressivity and frame definability of the logic of false belief. Moreover, due to the interdefinability of the operators of radical ignorance and false belief, we also axiomatized the logic of radical ignorance over the class of all frames and the class of serial and transitive frames. When translating the minimal logic of false belief to that of radical ignorance, we need to be cautious, since the translation only gives us an incomplete proof system, and one special axiom needed to be considered as well.

The almost definability schema is an important and useful tool in finding the suitable canoical relation and the desired core axioms for the special systems. Such usage has been made in the literature, see~\cite{Fanetal:2014,Fanetal:2015,Fan:2019,Fan:2021disjunctive}. This seems to be incomparable with other methods. We can try to extend such almost definability schema to other logics.

Coming back to the logics involved in this paper, one can explore the bisimulation notion for the logics of false belief and radical ignorance. Note that the almost definability schema is not enough for the notion of the bisimulation here, as in the case of the canonical relation. More things are needed to be taken account of. This is unlike the bisimulation of the contingency logic in the literature~\cite{Fanetal:2014}. 

Another future work is to axiomatize the logic of false belief over symmetric frames. As remarked before Sec.~\ref{sec.radicalignorance}, Almost Definability Schema also guides us to find an axiom AB, that is, $W\psi\land\neg\phi\to W((W\chi\to\neg W(\phi\land\chi))\land\psi)$ of $\mathcal{L}(W)$, which is valid over the class of symmetric frames. If we define the canonical model $\M^c$ for the system ${\bf B^W}$ (that is, the extension of ${\bf K^W}$ with the axiom AB) as in Def.~\ref{def.cm}, we can show that $\M^c$ is {\em almost} symmetric: let $s,t\in S^c$, if $sR^ct$ and $W\psi\in t$ for some $\psi$, then $tR^cs$. However, $\M^c$ is not symmetric, since if $W\psi\in t$ for no $\psi$ and $t\neq s$, then according to the definition of $R^c$, ${\sim}tR^cs$, even if $sR^ct$. Therefore, in order to axiomatize the symmetric logic of false belief, more work needs to be done.

%Future work: bisimulation

\bibliographystyle{plain}
\bibliography{biblio2024,biblio2016}

\end{document}